\documentclass[11pt, a4paper]{amsart}
\usepackage[a4paper, margin=3cm]{geometry}
\usepackage{amssymb,amsfonts}
\usepackage{tikz-cd}

\usepackage{color}
\newtheorem{Theorem}{Theorem}[section]

\newtheorem{Lemma}[Theorem]{Lemma}
\newtheorem{Corollary}[Theorem]{Corollary}

\newtheorem{Proposition}[Theorem]{Proposition}
\theoremstyle{definition}
\newtheorem{Definition}[Theorem]{Definition}
\newtheorem{Remark}[Theorem]{Remark}

\def\C{\mathbb C}

\def\N{\mathbb N}

\def\R{\mathbb R}

\newcommand{\CF}{\mathcal{F}}

\newcommand{\CH}{\mathcal{H}}

\newcommand{\CJ}{\mathcal{J}}

\newcommand{\CV}{\mathcal{V}}

\def\be{\begin{equation}}
\def\ee{\end{equation}}
\def\bt{\begin{Theorem}}
\def\et{\end{Theorem}}
\def\bi{\begin{itemize}}
\def\ei{\end{itemize}}
\def\bea{\begin{eqnarray}}
\def\eea{\end{eqnarray}}
\def\beast{\begin{eqnarray*}}
\def\eeast{\end{eqnarray*}}
\def\ben{\begin{enumerate}}
\def\een{\end{enumerate}}

\def\bi{\bibitem}

\newcommand{\norm}[1]{\left\Vert#1\right\Vert}

\def\lan{{\langle}}
\def\ran{{\rangle}}

\newcommand{\half}{{\frac{1}{2}}}

\renewcommand{\MR}[1]{} 

\newcommand{\abs}[1]{\left\vert#1\right\vert}

\allowdisplaybreaks[4]

\begin{document}
\title[On the commutants of generators]{On the commutants of  generators of $q$-deformed  Araki-Woods von Neumann algebras}

\author[Bikram]{Panchugopal Bikram}
\author[Mukherjee]{Kunal Mukherjee}


\address{School of Mathematical Sciences,
National Institute of Science Education and Research, Bhubaneswar, Utkal - 752050, India.}
\address{Department of Mathematics, IIT Madras, Chennai - 600036, India.}
\email{bikram@niser.ac.in, kunal@iitm.ac.in}


\begin{abstract} 
The generating abelian subalgebras arsing from vectors in the ergodic component of Hiai's construction of the $q$-deformed Araki-Woods von Neumann algebras are quasi-split.
\end{abstract}

\maketitle

\section{introduction}

The $q$-deformed Araki-Woods von Neumann algebras constructed by Hiai in \cite{Hiai} combining Shlyakhtenko's construction of free Araki-Woods factors in \cite{Shlyakhtenko} and Bo$\overset{.}{\text{z}}$ejko-Speicher's construction in \cite{BS} are complicated objects. While there has been substantial advancement towards the structure of free Araki-Woods factors in the recent years $($\cite{HSV17} being the latest$)$, very little is known about the  $q$-deformed Araki-Woods von Neumann algebras; even the fundamental question about factoriality of such algebras is open. The difficulty in studying the algebras constructed by Hiai is that there is minimal room to perform and control meaningful computations with the generators to develop insight. 

Hiai's construction associates a von Neumann algebra to a real Hilbert space equipped with an orthogonal representation of $\R$ and a parameter $-1<q<1$ in a way that interpolates between the bosonic and fermionic statistics.  In \cite{BM17}, we proved that the $q$-deformed  Araki-Woods von Neumann algebras are factors if the associated orthogonal representation either has a nonzero weakly mixing component or has a fixed vector in the real Hilbert space of dimension bigger than $1$. To prove this, we took advantage of the aforesaid fixed vector to show that the associated self-adjoint generator generates an algebra which posses faithful normal conditional expectation with respect to the vacuum state and that this subalgebra is a strongly mixing MASA living inside the centralizer of the vacuum state.  

MASAs are both classical and fundamental objects in studying von Neumann algebras. In most constructions of von Neumann algebras, it is possible to locate some canonical MASAs and study their properties which further entail structures of the ambient von Neumann algebras. Being unable to construct MASAs in a von Neumann algebra poses tough challenges and in this paper, we highlight that this indeed is the case for the von Neumann algebras of Hiai regardless of the parameter $-1<q<1$ or the initial orthogonal transformation unless such transformation has a fixed point. 

On the other hand, the free Araki-Woods factors of type $\rm{III}_1$ satisfy the Connes' bicentralizer problem \cite{Ho09}. Thus, by the results in \cite{HP17}, there exists a singular MASA in any free Araki-Woods factor of type $\rm{III}_1$ which is the range of a normal conditional expectation.   

In this paper, we show that for any normalized vector on the real Hilbert space associated with the ergodic component of the orthogonal representation, the inclusion of the associated abelian subalgebra inside the ambient von Neumann algebra is quasi-split. The same inclusion is split when the ambient von Neumann algebra is a type $\rm{III}$ factor. Thus, these abelian subalgebras have large relative commutants. The last statement is surprising and unexpected in the sense that there is no easy way to write some operators that commute with these generators.

This difficult nature of the generating abelian subalgebras poses a challenge to construct a single MASA arising from the ergodic component, let alone to have a conditional expectation. Though, we know for sure that such a MASA $($with expectation$)$ exist whenever the ambient algebra is a type $\rm{III}_1$ factor with $\abs{q}$ being small $($from the discussion above$)$ and  in some of the other cases when it is a type $\rm{III}_\lambda$ factor with $\lambda\in (0,1)$ $($see \cite{Po83}$)$.   

The organization of this paper is as follows. In \S\ref{Araki-Woods}, we collect all the necessary facts about the construction of Hiai that is needed to address the problem. To keep the paper self-contained, we make a short account on $($quasi$)$ split inclusions following \cite{BDL90, DL83, F02} in \S\ref{splitinclusionsection}. Finally, in \S\ref{Commutants}, the main section of this paper, we establish $($quasi$)$ split inclusions and qualitatively compute the relative commutants mentioned above.

\section{Construction and Basic Facts}\label{Araki-Woods}

In this section, we collect some facts about the $q$-deformed $($free$)$ Araki-Woods von Neumann algebras constructed by Hiai in \cite{Hiai} that will be indispensable for our purpose. For a detailed exposition, we refer the interested readers to \cite{Hiai,Shlyakhtenko}. This section is preliminary and has some overlap with \S2 and \S3 of \cite{BM17} to keep this paper self-contained. As a convention $($following \cite{Shlyakhtenko,Hiai}$)$, we assume that inner products are linear in the second variable. All von Neumann algebras considered in this paper have separable preduals, all Hilbert spaces are separable and all inclusions of von Neumann algebras are assumed to be unital.

Let $\CH_{\R}$ be a real Hilbert space and let $t\mapsto U_t$, $t\in \mathbb{R}$,
be a strongly continuous orthogonal representation of $\R$ on $\CH_{\R}$. 
Let  $\CH_\C=\CH_{\R}\otimes_\R \C$ denote the complexification of $\CH_{\R}$.
Denote the inner product and the norm on $\CH_\C $ by $\langle \cdot , \cdot \rangle_{\CH_\C} $ 
and $\norm{\cdot}_{\CH_\C}$ respectively. Identify $\CH_{\R}$ in $\CH_{\C}$ by $\CH_{\R} \otimes 1 $. 
Thus, $\CH_\C =  \CH_{\R} + i \CH_{\R}$, and as  a real Hilbert space the inner product of 
$\CH_{\R} $ in $\CH_\C $ is given by $\Re \langle \cdot, \cdot \rangle_{\CH_\C} $.
Consider the bounded anti-linear operator $\mathcal{J}: \CH_\C \rightarrow \CH_\C  $ given  by 
$\mathcal{J}(\xi + i \eta )= \xi - i \eta$, $\xi, \eta \in \CH_{\R} $,
and note that $\mathcal{J}\xi = \xi$ for $\xi\in \CH_{\R}$. Moreover,
\begin{align*}
\langle \xi, \eta \rangle_{\CH_\C} = \overline {\langle \eta , \xi \rangle}_{\CH_\C}  = \langle \eta , \mathcal{J} \xi \rangle_{\CH_\C}, \text{ for all } \xi \in \CH_{\C} , \eta \in \CH_{\R}.
\end{align*}
Linearly extend $t\mapsto U_t$ from $\CH_{\R} $ to a strongly continuous one parameter group of unitaries 
in $\CH_{\C}$  and denote the extensions by $U_t$ for each $t$ with abuse of notation.
Let $A$ denote the analytic generator. Then $A$ is positive,  nonsingular and self-adjoint. 
It is easy to see that $\CJ A = A^{-1}\CJ$. Introduce a new inner product on $\CH_{\C}$ by $\langle \xi, \eta \rangle_U = \langle \frac{2}{1+ A^{-1}} \xi, \eta \rangle_{\CH_\C}$, $\xi, \eta \in \CH_{\C}$, and let $\norm{\cdot}_{U}$ denote the associated norm on $\CH_{\C}$. Let $\CH$ denote the complex Hilbert space obtained by completing $(\CH_{\C}, \norm{\cdot}_{U})$.  The inner product and norm of $\CH$ will respectively be denoted by $\langle\cdot,\cdot\rangle_U$ and $\norm{\cdot}_{U}$ as well. Then, $(\CH_{\R}, \norm{\cdot}_{\CH_\C})\ni \xi \overset{\imath}\mapsto\xi \in (\CH_{\C}, \norm{\cdot}_{U})\subseteq (\CH,\norm{\cdot}_{U})$, is an isometric embedding of the real Hilbert space $\CH_\R$ in $\CH$ $($in the sense of \cite{Shlyakhtenko}$)$. 
With abuse of 
notation, identify $\CH_\R$ with its image $i(\CH_\R)$. Then, $\CH_{\R}\cap i\CH_{\R}=\{0\}$ and $\CH_{\R}+ i\CH_{\R}$ is dense in $\CH$ $($see pp. 332 \cite{Shlyakhtenko}$)$.

As $A$ is affiliated to $vN(U_t:t\in \R)$, note that 
\begin{align}\label{Liftisunitary}
\langle U_t\xi, U_t\eta\rangle_U=\langle \xi,\eta\rangle_U, \text{ for }\xi,\eta\in \CH_\C.
\end{align}
Consequently, $(U_t)$ extends to a strongly continuous unitary representation $(\widetilde{U}_t)$ of $\R$ on $\CH$. 
Let $\widetilde{A}$ be the analytic generator associated to $(\widetilde{U}_t)$, which is obviously an extension of $A$. From the definition of $\langle\cdot,\cdot\rangle_U$ on $\CH_\C$, it follows that if $\mu$ is the spectral measure of $A$, then $\nu=f\mu$ is the spectral measure of $\widetilde{A}$, where $f(x)=\frac{2x}{1+x}$ for $x\in\R_{\geq 0}$, and by the spectral theorem, the multiplicity functions in the associated direct integrals remain the same. Thus, we have the following:

\begin{Proposition}\label{Eigenvector}
Any eigenvector of $\widetilde{A}$ is an eigenvector of $A$ corresponding to the same eigenvalue. 
\end{Proposition}

Since the spectral information of $A$ and $\widetilde{A}$ $($and hence of $(U_t)$ and $(\widetilde{U}_t))$ are essentially the same, and $\widetilde{U}_t,\widetilde{A}$ are respectively extensions of $U_t,A$ for all $t\in \R$, we would now write 
$\widetilde{A}=A$ and $\widetilde{U}_t=U_t$ for all $t\in\R$. This abuse of notation will cause no confusion.

Following  \cite{BS}, the $q$-Fock space $\CF_q(\CH)$ of $\CH$ is constructed as follows for $ -1 < q < 1 $. Let $\Omega$ be a distinguished unit vector in $\C$ usually referred to as the vacuum vector. Denote $\CH^{\otimes 0}= \C\Omega$, and, for $n\geq 1$, let $\CH^{\otimes n} = \text{ span}_{\C}\{\xi_1 \otimes \cdots \otimes \xi_n : \xi_i \in \CH \text{ for }1\leq i\leq n\}$ denote the algebraic tensor products. Let $\CF_{fin}(\CH)=\text{ span}_{\C}\{\CH^{\otimes n}: n\geq 0 \}$. For $n,m\geq 0$ and $f=\xi_1 \otimes \cdots \otimes \xi_n \in \CH^{\otimes n}$, $g= \zeta_1\otimes \cdots \otimes \zeta_m\in \CH^{\otimes m}$, the association 
\begin{align}\label{qFock}
\langle f, g \rangle_q = \delta_{m, n} 
\sum_{ \pi \in S_n } q^{ i(\pi) } \langle \xi_1, \zeta_{\pi(1) }\rangle_U \cdots\langle \xi_n, \zeta_{\pi(n) } \rangle_U, 
\end{align} 
where $i(\pi)$ denotes the number of  inversions of the permutation $\pi \in S_n $, defines a positive definite sesquilinear form on  $ \CF_{fin}(\CH)$ and the $q$-Fock space $\CF_q(\CH) $ is the completion of $ \CF_{fin}(\CH)$ with respect to the norm $($denoted by $\norm{\cdot}_{q})$ induced by $ \langle \cdot, \cdot \rangle_q $. For $n\in\N$, let $\CH^{\otimes_q n}=\overline{\CH^{\otimes n}}^{\norm{\cdot}_q}$. 

The following formula of norm will be useful (c.f.  \cite{BKS}, \cite{BS},  and \cite{ER}):

If $\xi\in \CH$ and ${\norm{\xi}}_U = 1$, then 
\begin{align}\label{Normelt}
{\norm{\xi^{\otimes n}}}_q^2 = [n]_q!, 
\end{align}
where $[n]_q := 1+ q+ \cdots +  q^{(n-1)}$, $[n]_q! :=\prod_{j=1}^{n} [j]_q, \text{ for }n\geq 1$,    
and $[0]_q := 0$, $[0]_q! := 1$ by convention.

For $\xi\in \CH $, the left $q$-creation and $q$-annihilation operators on $\CF_q(\CH) $ are respectively defined by:
\begin{align}\label{Leftmult}
&c_q(\xi)\Omega  = \xi, \\ 
\nonumber&c_q(\xi) (\xi_1 \otimes \cdots\otimes \xi_n) =\xi \otimes  \xi_1 \otimes \cdots \otimes \xi_n, \text{ and},\\
\nonumber&c_q(\xi)^*\Omega  = 0,\\ 
\nonumber&c_q(\xi)^* (\xi_1 \otimes \cdots\otimes \xi_n) = \sum_{i = 1}^n{q^{i-1}}
\langle   \xi , \xi_i\rangle_U \xi_1 \otimes \cdots \otimes \xi_{i-1} \otimes \xi_{i+1} \otimes \cdots \otimes \xi_n, 
\end{align}
where $\xi_1 \otimes \cdots\otimes \xi_n\in \CH^{\otimes_q n}$ for $n\geq 1$.
The operators $c_q(\xi),c_q(\xi)^*\in\mathbf{B}(\CF_q(\CH))$ and they are 
adjoints of each other. 

Consider the $C^*$-algebra 
$\Gamma_q( \CH_\R, U_t) := C^{*}\{ s_q(\xi) : \xi \in \CH_\R \}$ and the von Neumann algebra
$\Gamma_q( \CH_\R, U_t)^{\prime\prime}$, where $s_q(\xi)  = c_q(\xi) + c_q(\xi)^*$, $\xi \in \CH_\R$.
This von Neumann algebra is known as the $q$-\textit{deformed Araki-Woods von Neumann algebra}  (see \cite[\S3]{Hiai}).  The vacuum state $\varphi_{q, U}:= 
\langle \Omega, \cdot\text{ } \Omega\rangle_q$ $($also called the $q$-\textit{quasi free state}$)$, is a faithful normal state of $\Gamma_q( \CH_\R, U_t)^{\prime\prime}$ and $\CF_{q}(\CH)$ is the GNS Hilbert space of $\Gamma_q( \CH_\R, U_t)^{\prime\prime}$ associated to $\varphi_{q,U}$. Thus, $\Gamma_q( \CH_\R, U_t)^{\prime\prime}$ acting on $\CF_{q}(\CH)$ is in standard form. We use the symbols $\langle \cdot,\cdot\rangle_q$ and $\norm{\cdot}_{q}$ respectively to denote the inner product and two-norm of elements of the GNS Hilbert space.

The modular theory of $\Gamma_{q}(\CH_{\R}, U_{t})^{\prime\prime}$ associated to $\varphi_{q,U}$ is as follows. Let $J_{\varphi_{q,U}}$ and $\Delta_{\varphi_{q,U}}$ respectively denote the modular conjugation and modular operator associated to $\varphi_{q,U}$ and let $S_{\varphi_{q,U}}=J_{\varphi_{q,U}}\Delta_{\varphi_{q,U}}^{\frac{1}{2}}$. The domain of $S_{\varphi_{q,U}}$ $($and hence of $\Delta_{\varphi_{q,U}}^{\frac{1}{2}})$ contains the complex span of simple tensors of all orders of vectors from $\CH_\R$, and, the operators $J_{\varphi_{q,U}}$ and $\Delta_{\varphi_{q,U}}$ depend on $A$. More precisely, for $n\in\N$, 
\begin{align}\label{modulartheory}
&J_{\varphi_{q,U}}(\xi_1 \otimes \cdots \otimes \xi_n) = A^{-1/2}\xi_n \otimes \cdots \otimes A^{-1/2}\xi_1, \text{ }\forall\text{ } \xi_{i}\in \mathcal{H}_{\mathbb{R}}\cap \mathfrak{D}(A^{-\half});\\
\nonumber&\Delta_{\varphi_{q,U}}(\xi_1 \otimes\cdots \otimes \xi_n ) =
A^{-1}\xi_1 \otimes \cdots\otimes A^{-1}\xi_n, \text{ }\forall\text{ } \xi_{i}\in \mathcal{H}_{\mathbb{R}}\cap \mathfrak{D}(A^{-1});\\
\nonumber&S_{\varphi_{q,U}}(\xi_1 \otimes \cdots \otimes \xi_n) =\xi_n \otimes \cdots \otimes \xi_1, \text{ }\forall\text{ } \xi_{i}\in \mathcal{H}_{\mathbb{R}}.
\end{align}
The modular automorphism group $(\sigma_{t}^{\varphi_{q,U}})$ of $\varphi_{q,U}$ is given by $\sigma_{-t}^{\varphi_{q,U}}=\text{Ad}(\CF(U_{t}))$, where $\CF(U_{t})=id\oplus \oplus_{n\geq 1} U_{t}^{\otimes_q n}$, for all $t\in \mathbb{R}$. In particular,
\begin{align}\label{modularaut}
\sigma^{\varphi_{q,U}}_{-t}(s_q(\xi)) = s_q(U_{t}\xi), \text{ for all } \xi \in \mathcal{H}_{\mathbb{R}}.
\end{align}

In this paper, the orthogonal representation remains arbitrary but fixed. Thus, to reduce notation, we will write $M_q=\Gamma_q(\CH_\R, U_t)^{\prime\prime}$ and $\varphi=\varphi_{q,U}$. We will also denote $J_{\varphi_{q,U}}$ by $J$ and $\Delta_{\varphi_{q,U}}$ by $\Delta$. 

We say that a vector $\xi\in \CH_\R$ is \textit{analytic}, if $s_{q}(\xi)$ is analytic for $(\sigma_t^{\varphi})$.

\begin{Lemma}$($\cite[Lemma 3.1]{BM17}$)$\label{VectorinMq}
The vector $\xi_{1}\otimes\cdots\otimes \xi_{n}\in M_q\Omega $ for any $\xi_i\in\CH_\R$, $1\leq i\leq n$ and $n\in \mathbb{N}$.
\end{Lemma}

Fix $\xi\in \CH_\R$ with $\norm{\xi}_{U}=1$. By Eq. $(1.2)$ of \cite{Hiai}, the moments of the operator $s_q(\xi)$ with respect to the $q$-quasi free state $\varphi(\cdot) = \lan \Omega, \cdot \Omega \ran_q$ are given by 
\begin{equation*}
\varphi(s_q(\xi)^n)=
   \begin{cases}
     0, & \text{if } n \text{ is odd}, \\
     \sum_{\CV = \{ \pi(r), \kappa(r) \}_{ 1 \leq r \leq \frac{n}{2} } }  q^{c(\CV)}, & \text{if }n\text{ is even},
  \end{cases}
\end{equation*}
where the summation is taken over all pair partitions $\CV = \{ \pi(r), \kappa(r) \}_{ 1 \leq r \leq \frac{n}{2} }$ of $\{1, 2, \cdots, n\}$ with $\pi(r)<\kappa(r) $ and $c(\CV)$ is the number of crossings of $\CV$, i.e., 
\begin{align*}
c(\CV) = \# \{ (r, s): \pi(r)< \pi(s) < \kappa(r) < \kappa(s) \}.
\end{align*}
The distribution of $s_q(\xi)$  does not depend on the group $(U_t )$. This distribution obeys the semicircular law $\nu_q$ which is absolutely continuous with respect to the uniform measure  supported on the interval $[-\frac{2}{\sqrt{ 1-q}}, \frac{2}{\sqrt{ 1-q}}]$; the associated orthogonal polynomials are $q$-Hermite polynomials $H_{n}^{q}$, $ n\geq 0$. Thus, $M_{\xi}=vN(s_q(\xi))$ is diffuse and $\{H_n^q(s_q(\xi))\Omega:n\geq 0\}$, is a total orthogonal set of vectors in $\overline{M_{\xi}\Omega}^{\norm{\cdot}_q}$. By convention, $\xi^{\otimes 0}=\Omega$. Note that $\overline{M_\xi\Omega}^{\norm{\cdot}_q}= \overline{\text{ span }\{\xi^{\otimes n}:n\geq 0}\}^{\norm{\cdot}_q}$, as $\xi^{\otimes n}=H_n^q(s_q(\xi))\Omega$ for all $n\geq 0$. 

It is to be noted that for $\xi\in \CH_\R$ with $\norm{\xi}_{U}=1$, there does not exist any $\varphi$-preserving faithful normal 
conditional expectation $($even appropriate operator valued weight$)$ on $M_\xi$ unless $U_t\xi=\xi$ for all $t\in \R$ $($see Thm. 4.2 \cite{BM17}$)$.

\section{Split Inclusions}\label{splitinclusionsection}

In this section, we study $($quasi$)$ split inclusions of von Neumann algebras that will be used as an auxiliary tool to investigate the $q$-deformed Araki-Woods von Neumann algebras. For inclusions of factors such study was carried out in \cite{BDL90} and quasi-split inclusions were studied in \cite{F02}. Interested readers are also referred to \cite{DL83}, where the authors discuss relations between `split property' and `quasi-innerness' of flip automorphisms on tensor products.  

In \cite{BDL90}, the authors work with factors, though \emph{some} of their proofs and results check out verbatim for general von Neumann algebras. For our purpose, we collect some of the results of \cite{BDL90, DL83, F02} that we need in the case when the von Neumann algebras \emph{may not} be factors. Thus, we do not claim originality of statements in this section. We need some preparation. 

Let a von Neumann algebra $M$ be represented in standard form on the GNS Hilbert space $\CH_\varphi:=L^2(M,\varphi)$ with respect to a faithful normal state $\varphi$. When there is no danger of confusion, we will write $L^2(M)$ for $L^2(M,\varphi)$ to simplify notation. Let $J_\varphi$, $\Delta_\varphi$ respectively denote the Tomita's conjugation and modular operators, and let $\Omega_\varphi$ denote the standard vacuum vector. The inner product and norm on $\CH_\varphi$ is denoted by $\langle\cdot,\cdot\rangle_\varphi$ and $\norm{\cdot}_{2,\varphi}$ respectively. Let $L^1(M):=M_*$ denote the predual of $M$. Let $M^{op}$ denote the opposite algebra of $M$ and $y^{op}:=J_\varphi y^* J_\varphi$ for all $y\in M$.

We have the following natural embeddings given by:
\begin{align}\label{embedding}
&\Phi_1: M\rightarrow L^1(M) \text{ by } \Phi_1(x)=\langle \cdot\Omega_\varphi, J_\varphi x\Omega_\varphi\rangle_\varphi, \text{ }x\in M;\\
&\nonumber\Phi_2:M\rightarrow L^2(M)\text{ by } \Phi_2(x)=\Delta_\varphi^{\frac{1}{4}}x\Omega_\varphi, \text{ }x\in M.
\end{align}
It is easily seen that $\Phi_p$, $p=1,2$, have norm $1$ and are $(w^*,w^*)$-continuous $($i.e., they are normal$)$. In fact, $\Phi_2$ is $(\sigma$-strong$^*,\norm{\cdot}_{2,\varphi})$-continuous. 

For $p=1,2$, let $\mathfrak{P}_\varphi^p\subseteq L^p(M)$ denote the standard cone of positives. Thus,  $\mathfrak{P}_\varphi^1=M_*^+$ and $\mathfrak{P}_\varphi^2=\overline{\Delta^{\frac{1}{4}}_\varphi M_+\Omega_\varphi}^{\norm{\cdot}_{2,\varphi}}$. For $n\in\N$, consider the faithful normal state $\varphi^{(n)}:=\varphi\otimes tr_n$ on $M^{(n)}:=M\overline\otimes M_n(\C)$, where $tr_n$ is the normalized trace on $M_n(\C)$.  

\begin{Definition}\label{positivemap}
Let $N,M$ be von Neumann algebras and let $M$ be equipped with a faithful normal state $\varphi$.  For $p=1,2$, a linear map $\Theta_p: N\rightarrow L^p(M)$ is said to be completely positive if $(\Theta_p\otimes id)(\Big(N\overline{\otimes}M_n(\C)\Big)_+)\subseteq \mathfrak{P}_{\varphi^{(n)}}^p$.  
\end{Definition}

Then, $\Phi_p$, $p=1,2$, are completely positive $($c.p. in the sequel$)$. 
Indeed, let $1_n$ denote the identity operator on the Hilbert space $M_n(\C)$, the inner product on $M_n(\C)$ being induced by $tr_n$. Then, $\Delta_{\varphi^{(n)}}=\Delta_{\varphi}\otimes 1_n$. Elements of
$L^2(M^{(n)},\varphi^{(n)})$ will be denoted as $[\zeta_{ij}]$ with $\zeta_{ij}\in L^2(M)$ for all $1\leq i,j\leq n$. Therefore, 
\begin{align*}
\Delta_{\varphi^{(n)}}^{\frac{1}{4}}[x_{ij}]\Omega_{\varphi^{(n)}}&= [\Delta_{\varphi}^{\frac{1}{4}}x_{ij}\Omega_\varphi]\\
&=(\Phi_2\otimes id)[x_{ij}], \text{ }[x_{ij}]\in M^{(n)}. 
\end{align*}
It readily follows that, if $0\leq [x_{ij}]\in M^{(n)}$, then $(\Phi_2\otimes id)[x_{ij}]\in \mathfrak{P}_{\varphi^{(n)}}^2$. 

Let $\{e_{i,j}:1\leq i,j\leq n\}$ denote the standard matrix units of $M_n(\C)$. Again, let $[a_i^*a_j], [x_i^*x_j]\in M^{(n)}$, where $a_i,x_i\in M$ for $1\leq i\leq n$. Then, $[a_i^*a_j], [x_i^*x_j]\geq 0$. Therefore, $[\langle \cdot\Omega_\varphi, J_\varphi x_i^*x_j\Omega_\varphi\rangle_\varphi]\in \Big(M\overline\otimes M_n(\C)\Big)_*$, and, 
\begin{align*}
\Big([\langle \cdot\Omega_\varphi, J_\varphi x_i^*x_j\Omega_\varphi\rangle_\varphi]\Big)([a_i^*a_j])&=\Big(\sum_{i,j}\langle \cdot\Omega_\varphi, J_\varphi x_i^*x_j\Omega_\varphi\rangle_\varphi\otimes e_{i,j}\Big)\Big(\sum_{l,m}a_l^*a_m\otimes e_{l,m}\Big)\\
&=\sum_{i,j}\sum_{l,m}\Big(\langle \cdot\Omega_\varphi, J_\varphi x_i^*x_j\Omega_\varphi\rangle_\varphi\otimes e_{i,j}\Big)(a_l^*a_m\otimes e_{l,m})\\
&=\sum_{i,j}\sum_{l,m} \langle a_l^*a_m\Omega_\varphi, J_\varphi x_i^*x_j\Omega_\varphi\rangle_\varphi e_{i,j}(e_{l,m})\\
&=\sum_{i,j}\sum_{l,m} \langle a_l^*a_m\Omega_\varphi, J_\varphi x_i^*x_j\Omega_\varphi\rangle_\varphi tr_n(e_{i,j}e_{l,m}) \\
&\indent\indent\indent\indent\indent\indent\indent\indent\indent\indent\indent\Big(e_{i,j}\mapsto \langle\cdot, J_{tr_n}e_{i,j}\Omega_{tr_n}\rangle_{tr_n}\Big) \\
&=\sum_{l,m} \langle a_l^*a_m\Omega_\varphi, J_\varphi x_m^*x_l\Omega_\varphi\rangle_\varphi \\
&=\sum_{l,m} \langle a_l^*a_m\Omega_\varphi, J_\varphi x_m^*J_\varphi  J_\varphi  x_l J_\varphi \Omega_\varphi\rangle_\varphi \\
&=\sum_{l,m} \langle J_\varphi x_m J_\varphi  a_m\Omega_\varphi,   J_\varphi  x_l J_\varphi a_l \Omega_\varphi\rangle_\varphi \\
&= \norm{\sum_l J_\varphi  x_l J_\varphi a_l \Omega_\varphi}^2_{2,\varphi}\geq 0.
\end{align*}
Since $(M\overline\otimes M_n(\C))_+$ is the sum of $n$ positive elements of the form $[x_i^*x_j]$ with $x_i\in M$ for $1\leq i\leq n$ \cite[Lemma 3.13]{Pa03}, it follows that $(\Phi_1\otimes id)\Big((M\overline\otimes M_n(\C))_+\Big)\subseteq \mathfrak{P}_{\varphi^{(n)}}^1$.

All throughout this section, $B\subseteq M$ will be a unital von Neumann subalgebra of $M$. 

\begin{Definition}\label{split}
\noindent
\begin{enumerate}
\item The inclusion $B\subseteq M$ is said to be \emph{split}, if there exists a type $\rm{I}$ factor $F$ such that $B\subseteq F\subseteq M$. 
\item The inclusion $B\subseteq M$ is said to be \emph{quasi-split}, if the map 
\begin{align*}
B\otimes_{alg} M^{op}\ni b\otimes y^{op}\mapsto aJ_\varphi y^*J_\varphi\in \mathbf{B}(\mathcal{H}_\varphi),
\end{align*}
extends to a normal $*$-homomorphism $\eta$ of $B\overline\otimes M^{op}$ $($acting on $\mathcal{H}_\varphi\otimes \mathcal{H}_\varphi)$ onto $B\vee M^\prime$. 
\end{enumerate}
\end{Definition} 
By the uniqueness of standard form, it follows that the quasi-split property is really a property of the inclusion. 

Even in the simplest case, the homomorphism $\eta$ will fail to be an isomorphism. For example, consider $A=C(\{0,1\})\simeq \C^2$ acting on $\C^2$ as multiplication operators. Obviously, $A^\prime = A=A^{op}$. The homomorphism associated to the inclusion $A\subseteq A$ is $\eta: A\otimes A^{op}\rightarrow A$ given by $\eta(f)(t) = f(t,t)$, $f\in A\otimes A^{op}$ and $t \in \{0,1\}$. Then $Ker(\eta)$ is non-trivial, so $\eta$ is not an isomorphism and the inclusion $A\subseteq A$ is quasi-split.

\begin{Definition}\cite{BDL90}\label{extendable}
Let $N$ and $M$ be von Neumann algebras and let $p=1,2$. A normal c.p. map $\Phi_p: N\rightarrow L^{p}(M)$, is said to be extendable, if for any von Neumann algebra $\widetilde{N}$ with separable predual containing $N$, there exists a normal c.p. map $\widetilde{\Phi}_p :\widetilde{N}\rightarrow L^p(M)$, which extends $\Phi_p$. 
\end{Definition}

\begin{Proposition}\cite[Prop. 2]{F02}\label{Extend1and2equivalent}
Let $B\subseteq M$ be an inclusion of von Neumann algebras, where $M$ is represented in standard form on the GNS Hilbert space $\CH_\varphi$ with respect to a faithful
normal state $\varphi$. Then the following are equivalent.
\begin{enumerate}
\item ${\Phi_1}_{\upharpoonleft B}: B\rightarrow L^1(M)$ is extendable.
\item  ${\Phi_2}_{\upharpoonleft B}: B\rightarrow L^2(M)$ is extendable.
\item $B\subseteq M$ is a quasi-split inclusion.\footnote{The term `\emph{quasi-split}' has also been used by Doplicher and Longo in the same context to mean that $\eta$ is a $*$-isomorphism, which implies but need not be implied by $(1)$ or $(2)$ of Prop. \ref{Extend1and2equivalent} \cite[Defn. 1.4]{DoL84}. Also see \cite[Thm. 1]{DL83}.}
\end{enumerate}
\end{Proposition}

\begin{Remark}
The abstract characterization of quasi-split inclusions above is practically impossible to verify in most situations. Thus, to check if an inclusion is quasi-split, one needs more applicable conditions.
\end{Remark}

\begin{Definition}\label{nuclearmaps}
Let $X$ and $Y$ be two Banach spaces and let $ \Psi : X \rightarrow Y$ be a bounded linear map. Then, $\Psi$ is called a \emph{nuclear} map if and only if there exist sequences $x_n^* \in X^* $ and $y_n \in Y$ such that $\sum_{n=1}^\infty  \norm{x_n^*} \norm{y_n} < \infty $ and 
\begin{align*}
\Psi(x) = \sum_{n=1}^\infty x_n^*(x)y_n, \text{ for all }  x \in X. 
\end{align*}
\end{Definition}
Clearly, the collection of nuclear maps from $X$ to $Y$ form a subspace of $\textbf{B}(X,Y)$ which can be identified with the projective tensor product $X^*\hat\otimes Y$. Thus, these maps inherit a natural norm $\norm{\cdot}_n$, called the nuclear norm, and $\norm{\Psi}_n=\inf\{\sum_{n=1}^\infty  \norm{x_n^*} \norm{y_n} : \Psi(x) = \sum_{n=1}^\infty x_n^*(x)y_n,  \text{ for all }  x \in X\}$. Moreover, when $X$ in Defn. \ref{nuclearmaps} is a von Neumann algebra and $\Psi$ is continuous with respect to the strong$^*$ topology on $X$ and weak topology on $Y$, then one can ensure a choice to represent $\Psi$ such that $x_n^*\in X_{*}$ for all $n$. Note that nuclear maps between appropriate Banach spaces have very close connections with quasi-split inclusions. 

The next result is a combination of statements from \cite{BDL90} and \cite{F02}. However, we provide a rough sketch of the proof for the sake of convenience. 

\begin{Proposition}\label{nuclearimplysplit}
Let $B\subseteq M$ be an inclusion of von Neumann algebras, where $M$ is represented in standard form on the GNS Hilbert space $\CH_\varphi$ with respect to a faithful
normal state $\varphi$. Then, 
\begin{align*}
{\Phi_p}_{\upharpoonleft B} \text{ is nuclear } \text{ }\Rightarrow \text{ }&{\Phi_p}_{\upharpoonleft B} \text{ is extendable }\text{ }\Leftrightarrow B\subseteq M \text{ is quasi-split}, \text{ }p=1,2.
\end{align*}
\end{Proposition}

\begin{proof}
Let $\Phi_2^\prime: L^2(M)\rightarrow L^1(M)$ denote the transpose of $\Phi_2$, i.e., $\Phi_2^\prime(\xi) =\langle \Phi_2(\cdot), J_\varphi\xi\rangle_{2,\varphi}$, $\xi\in L^2(M)$. We have the following commutative diagram. 
\begin{center}
\begin{tikzcd}
M \arrow[rr, "\Phi_1"] \arrow[rd, "\Phi_2"] &  & L^1(M) \\
 & L^2(M) \arrow[ru, "\Phi_2^\prime"] &       
\end{tikzcd}
\end{center}
Indeed, for $x,y\in M$, one has
\begin{align*}
\Big((\Phi_2^\prime\circ\Phi_2)(x)\Big)(y) &=\Big(\Phi_2^\prime(\Delta_\varphi^{\frac{1}{4}}x\Omega_\varphi)\Big)(y)\\
&=\langle \Delta_\varphi^{\frac{1}{4}}y\Omega_\varphi, J_\varphi\Delta_\varphi^{\frac{1}{4}} x\Omega_\varphi\rangle_{2,\varphi}\\
&=\langle \Delta_\varphi^{\frac{1}{4}}y\Omega_\varphi, \Delta_\varphi^{-\frac{1}{4}} J_\varphi x\Omega_\varphi\rangle_{2,\varphi}\\
&=\langle y\Omega_\varphi, J_\varphi x\Omega_\varphi\rangle_{2,\varphi}\\
&=\Big(\Phi_1(x)\Big)(y).
\end{align*}
\indent Now suppose that ${\Phi_2}_{\upharpoonleft B}$ is nuclear. Then, ${\Phi_1}_{\upharpoonleft B}$ is also nuclear. Then, the first part of the proof of \cite[Prop. 2.3]{BDL90} entails that ${\Phi_1}_{\upharpoonleft B}$ is extendable. Thus, ${\Phi_2}_{\upharpoonleft B}$ is also extendable and this is equivalent to the inclusion 
$B\subseteq M$ being quasi-split by Prop. \ref{Extend1and2equivalent}.
\end{proof}

\begin{Proposition}\label{Factorimplysplit}
Let $B$ be a unital von Neumann subalgebra of a factor $M$, where $M$ is represented in standard form on the GNS Hilbert space $\CH_\varphi$ with respect to a faithful normal state $\varphi$. If the inclusion $B\subseteq M$ is quasi-split, then there exists a type $\rm{I}$ factor $F$ such that $B\otimes 1\subseteq F \subseteq M\overline\otimes \mathbf{B}(\mathcal{K})$, with $\dim(\mathcal{K})=\aleph_0$.  Further, if $M$ is a type $\rm{III}$ factor, then $B\subseteq M$ is a split inclusion. 
\end{Proposition}

\begin{proof}
We are given that $B\otimes_{alg} M^{op}\ni b\otimes y^{op}\overset{\eta_0}\mapsto bJ_\varphi y^* J_\varphi\in B\vee M^{\prime}\subseteq \mathbf{B}(\mathcal{H}_\varphi)$ extends to a surjective normal $*$-homomorphism $\eta: B\overline\otimes M^{op}\subseteq \mathbf{B}(\mathcal{H}_\varphi)\overline\otimes  \mathbf{B}(\mathcal{H}_\varphi)\rightarrow B\vee M^{\prime}\subseteq  \mathbf{B}(\mathcal{H}_\varphi)$. Then, $Ker(\eta)$ is a two-sided weakly closed ideal in $B\overline\otimes M^{op}$. Thus, there exists a projection $p\in \mathcal{Z}(B\overline\otimes M^{op})$ such that $Ker(\eta)=(B\overline\otimes M^{op})p$. Since $M$ is a factor, so is $M^{op}$. By Tomita's theorem on commutants, it follows that $\mathcal{Z}(B\overline\otimes M^{op})=\mathcal{Z}(B)\overline\otimes \C1_{M^{op}}$. Consequently, there exists a projection $p_0\in \mathcal{Z}(B)$ such that $p=p_0\otimes 1_{M^{op}}$. 

Therefore, $\eta(p_0\otimes 1_{M^{op}})=0$ and hence $\eta_0(p_0\otimes 1_{M^{op}})=0$. However, since $M$ is a factor, $\eta_0$ is injective by a well known result $($see \cite[Prop. 1.20.5]{Sa}$)$. This contradiction forces that $p_0=0$, forcing $\eta$ to be a $*$-isomorphism. Thus, $B\otimes 1\subseteq M\overline\otimes \mathbf{B}(\mathcal{K})$ is a split inclusion by \cite[Thm. 1]{DL83}. 

If $M$ is a type $\rm{III}$ factor, then the result follows directly from \cite[Thm. 1, Cor. 1]{DL83}. 
\end{proof}

Again, by Tomita's fundamental theorem on commutants $($see \cite[pp. 258]{StZs}$)$ we have:

\begin{Lemma}\cite[Lemma 2]{DL83}\label{splitrepresentation}
Let $B\subseteq M\subseteq \textbf{B}(\mathcal{H}_\varphi)$ be von Neumann algebras. Then the following are equivalent.
\begin{enumerate}
\item  The inclusion $B\subseteq M$ is split. 
\item There exist Hilbert spaces $\CH_1$ and $\CH_2$ and faithful normal representations $\pi_B:B\rightarrow \textbf{B}(\CH_1)$ and $\pi_{M^\prime}:M^\prime\rightarrow \textbf{B}(\CH_2)$ such that $xy^\prime\mapsto \pi_B(x)\otimes\pi_{M^\prime}(y^\prime)$, $x\in B$ and $y^\prime\in M^\prime$, extends to a \textbf{spatial} isomorphism between $B\vee M^\prime$ and 
$\pi_B(B)\overline{\otimes}\pi_{M^\prime}(M^\prime)$. Moreover, $B^\prime\cap M\cong \pi_B(B)^\prime\overline{\otimes}(\pi_{M^\prime}(M^\prime))^\prime$. 
\end{enumerate}
Further, if $M$ is of type $\rm{III}$ and $B\subseteq M$ is a split inclusion, then $B^\prime\cap M$ is of type $\rm{III}$. 
\end{Lemma}

\begin{Remark}\label{Quasisplitingeneral}
\noindent 
\begin{enumerate}
\item If $B\subseteq M$ be a quasi-split inclusion as in Defn. \ref{split}, then there exists a projection $0\neq p\in \mathcal{Z}(B\overline\otimes M^{op})$ such that $(B\overline\otimes M^{op})p\simeq B \vee M^{\prime}$. Further, if $M$ is of type $\rm{III}$ then so is $B \vee M^{\prime}$. Therefore, the isomophism is spatial \cite[Cor. 8.13]{StZs}. Consequently, on taking commutants we have, $B^\prime\cap M \simeq (B_0\overline\otimes M)p$, where $B_0$ is the commutant of $B$ taken in $\mathcal{H}_\varphi$. Thus, $B^\prime\cap M$ is of type $\rm{III}$. 

\item Let $B\subseteq M$ be a quasi-split inclusion such that $B$ is a MASA in $M$. Then $B\vee M^\prime$ is a type $\rm{I}$ algebra, as $B=B^\prime\cap M = (B\vee M^\prime)^\prime$ is abelian. Consequently, $(B\overline\otimes M^{op})p$ is of type $\rm{I}$ too, and thus $B\overline\otimes M^{op}$ has a non-trivial type $\rm{I}$ central summand. This follows from the previous analysis. However, more is true. We claim that $B\subseteq M$ is isomorphic to a  direct sum of MASAs  in type $\rm{I}$ factors.

First, note that if $0\neq p_0\in \mathcal{Z}(M)$, then by Defn. \ref{split} it follows that $Bp_0\subseteq Mp_0$ is also a quasi-split inclusion. Since $\mathcal{Z}(M)\subseteq B$, so $\mathcal{Z}(M)\subseteq M$ is also a quasi-split inclusion with associated homomorphism $\eta:\mathcal{Z}(M)\overline\otimes M^{op}\rightarrow \mathcal{Z}(M)\vee M^\prime$. Let $e_{\mathcal{Z}(M)}$ denote the Jones' projection associated to $\mathcal{Z}(M)$. Then, by compressing $\eta_{\upharpoonleft\mathcal{Z}(M)\overline\otimes\mathcal{Z}(M) }$ with $e_{\mathcal{Z}(M)}$, it follows that $\mathcal{Z}(M)\subseteq \mathcal{Z}(M)$ is a quasi-split inclusion. Consequently, $\mathcal{Z}(M)$ is completely atomic. Therefore, $M$ is a direct sum of factors. This along with the initial argument establishes the claim.   
\end{enumerate}
\end{Remark}

\begin{Corollary}\label{largerelativecommutant}
If $B\subseteq M$ be a quasi-split inclusion of von Neumann algebras where $M$ is of type $\rm{III}$, then $B^\prime\cap M$ is of type $\rm{III}$. 
\end{Corollary}

\section{The relative commutant of $M_\xi$ }\label{Commutants}

In this section, we will measure the relative commutant of some generating abelian subalgebras of $M_q$ and establish that there is no easy way to construct MASAs in $M_q$ other than considering the generators that live inside the centralizer $M_q^\varphi$ $($see \cite{BM17}$)$. Our results are valid for all $-1<q<1$ and complete when $M_q$ is a factor.

We believe that rather than providing a general proof that takes into account both the almost periodic and weak mixing components of the associated orthogonal representation, it is more illuminating to 
work out the details in the case of a single $2\times 2$ almost periodic component of the orthogonal representation. In fact, this is the way one would find this proof. 

Since $t\mapsto U_t$, $t\in \R$, is a strongly continuous orthogonal representation of $\R$ on the real Hilbert space $\CH_\R$, there is a unique decomposition $($c.f. \cite{Shlyakhtenko}$)$,
\begin{align}\label{Representation}
( \CH_\R, U_t)  =  
\left(\bigoplus_{j = 1}^{N_1} (\R, \text{id} )\right) \oplus \left( \bigoplus_{k=1}^{N_2}(\CH_\R(k), U_t(k) ) \right) \oplus (\widetilde{\CH}_\R, U^{wm}_t),
\end{align}
where $0\leq N_1,N_2\leq \aleph_0$, 
\begin{align}\label{EvenTranform}
\CH_\R(k) = \R^2, \quad U_t(k)  = \left( \begin{matrix} \cos( t\log \lambda_k)& - \sin( t\log\lambda_k) \\
\sin( t\log\lambda_k)& \cos( t\log\lambda_k)  \end{matrix} \right), \text{ }\lambda_k > 1,
\end{align}
and $(\widetilde{\CH}_\R,U^{wm}_t)$ corresponds to the weakly mixing component of the orthogonal representation; thus $\widetilde{\CH}_\R$ is either $0$ or infinite dimensional. 
In this section,  we assume that $N_2\neq 0$ or $\widetilde{\CH}_\R\neq 0$. 

First assume $N_2\neq 0$. Let $\xi_{2k-1}=0\oplus\cdots\oplus 0\oplus\left( \begin{matrix}1\\0 \end{matrix}\right)\oplus 0\oplus\cdots \oplus 0\in\bigoplus_{k=1}^{N_2}\CH_\R(k)$ and $\xi_{2k} =0\oplus\cdots\oplus 0\oplus\left( \begin{matrix}0\\1\end{matrix}\right)\oplus 0\oplus\cdots \oplus 0\in \bigoplus_{k=1}^{N_2}\CH_\R(k)$ be vectors with nonzero entries in the $k$-th position for $1\leq k\leq N_2$. Denote 
\begin{align*}
\zeta_{2k-1}= \frac{\sqrt{\lambda_k +1}}{2}(\xi_{2k-1}+i\xi_{2k}) \text{ and }\zeta_{2k}=\frac{\sqrt{{\lambda}^{-1}_k+1}}{2}(\xi_{2k-1}-i\xi_{2k}).
\end{align*}
Thus, $\zeta_{2k-1},\zeta_{2k} \in\CH_\R(k)+i\CH_\R(k)$ form an orthonormal basis of $(\CH_\R(k)+i\CH_\R(k),\langle\cdot,\cdot\rangle_{U})$ for $1\leq k\leq N_2$.  The analytic generator $A(k)$ of $(U_t(k))$ is given by 
\begin{align*}
A(k)=\frac{1}{2}\left( \begin{matrix} \lambda_k + \frac{1}{\lambda_k} & i(\lambda_k - \frac{1}{\lambda_k})\\
-i(\lambda_k -\frac{1}{ \lambda_k}) &\lambda_k + \frac{1}{\lambda_k}
\end{matrix}\right), \text{ }1\leq k\leq N_2.
\end{align*}
Moreover, 
\begin{align*}
A(k)\zeta_{2k-1} = \frac{1}{\lambda_k}\zeta_{2k-1}  \text{ and }  A(k)\zeta_{2k} = \lambda_k\zeta_{2k}, \text{ }1\leq k\leq N_2.
\end{align*}

Fix $\tilde{k}$, with $1\leq \tilde{k}\leq N_2$, and rename the pair $(\xi_{2\tilde{k}-1}, \xi_{2\tilde{k}} )= (\xi_0,\xi_0^{\prime})$ to distinguish it from  other pairs. 
Accordingly, let $\CH_\R(0)=\CH_\R(\tilde{k})$, $A(0)=A(\tilde{k})$, $\lambda_0=\lambda_{\tilde{k}}$ and $(\zeta_0,\zeta_0^\prime)=(\zeta_{2\tilde{k}-1},\zeta_{2\tilde{k}})$. 

In \cite{BM17}, we established that if $N_1\geq 1$, and $\varsigma_l=0\oplus\cdots\oplus 0\oplus 1\oplus 0\oplus\cdots \oplus 0$, where $1$ appears at the $l$-th position,
then $M_{\varsigma_l}\subseteq M_q^\varphi$ is a MASA in $M_q$. Here, we investigate the relative commutant of $M_{\xi_0}$ and $M_{\xi_0^\prime}$. In contrast with the results in \cite{BM17}, one might expect that $M_{\xi_0}$ and hence by symmetry $M_{\xi_0^\prime}$ too would be MASAs in $M_q$; but we show that $M_{\xi_0}\subseteq M_q$ $($and hence $M_{\xi_0^\prime}\subseteq M_q)$ is a quasi-split inclusion, which can be  surprising or unexpected because exhibiting operators in $M_q$ by hand that commutes with $M_{\xi_0}$ is very hard. We only work with $M_{\xi_0}$, as the analysis in the case of $M_{\xi_0^\prime}$ is analogous. 

We will consider the restriction of the symmetric embedding  ${\Phi_2}_{\upharpoonleft M_{\xi_0}} : M_{\xi_0} \rightarrow L^2(M_q,\varphi)$. As noted in \S\ref{splitinclusionsection},
${\Phi_2}_{\upharpoonleft M_{\xi_0}} $ is $\sigma$-strong$^*$ to $\norm{\cdot}_q$ continuous map of norm $1$. 

We now explore the behavior of $\Delta$.   Denote  
\begin{align*}
e_0 = \frac{1} {\sqrt{2}}(\xi_0+i\xi_{0}')
\text{ and }
e_0^\prime = \frac{1} {\sqrt{2}}(\xi_0-i\xi_{0}'). 
\end{align*}
Note that $e_0,e_0^\prime$ are respectively scalar multiplies of $\zeta_0$ and $\zeta_0^\prime$ but are \emph{not} unit vectors with respect to $\norm{\cdot}_U$, though they are orthonormal vectors with respect to $\norm{\cdot}_{\CH_\C}$. Nevertheless, $A(0) e_0 =\frac{1}{\lambda_0} e_0 $ and $A(0) e_0' = \lambda_0  e_0' $. Further, note that, 
\begin{align}\label{linearaddition}
\xi_0 = \frac{1} {\sqrt{2}}(e_0 + e_{0}')\text{ and } \xi_0^\prime =\frac{i} {\sqrt{2}}(e_{0}'-e_0).
\end{align}

Let $W: \C^2 \rightarrow \C^2 $ be the unitary $($with respect to $\langle\cdot,\cdot\rangle_{\CH_\C})$ such that $W\xi_0 = e_0 $ and $W\xi_0' = e_0'$. Then, note that 
\begin{align}\label{diagonalize}
W^*A(0)W= \left( \begin{matrix} \frac{1}{\lambda_0} & 0\\
0 &  \lambda_0 \end{matrix}\right). 
\end{align}

The following calculations are crucial for the analysis.

\begin{Lemma}\label{normdelta} Let $ \alpha,\beta \in \R$ and $z=\alpha+i\beta$, then 
$\norm{\Delta^z\xi_0}_q= \sqrt{ \frac{\lambda_0^{2\alpha} +\lambda_0^{1-2\alpha}}{1+\lambda_0} } $.

\end{Lemma}\label{Deltaonxi0}
\begin{proof} 
Note that from Prop. \ref{Eigenvector} and Eq. \eqref{modulartheory}, it follows that  $s_q(\xi_0)$ is analytic with respect to $(\sigma_t^\varphi)$. Observe that
\begin{align*}
\norm{\Delta^\alpha \xi_0}_q^2 &= \langle \Delta^\alpha \xi_0, \Delta^\alpha \xi_0 \rangle_q\\
& = \langle \Delta^{2\alpha}\xi_0,\xi_0 \rangle_q\\
&= \langle \frac{2A}{1+A} A^{-2\alpha} \xi_0,  \xi_0 \rangle_{\CH_\C} \text{ (by Prop. \ref{Eigenvector} and Eq. \eqref{modulartheory})} \\
&= \langle \frac{2A(0)}{1+A(0)} {A(0)}^{-2\alpha} \xi_0,  \xi_0 \rangle_{\CH_\C} \\
&= \langle \frac{2{A(0)}^{1-2\alpha}}{1+A(0)} \xi_0,   \xi_0 \rangle_{\CH_\C} \\
&= \langle WW^* \frac{2{A(0)}^{1-2\alpha}}{1+A(0)} WW^*\xi_0,   \xi_0 \rangle_{\CH_\C}\\
&= \Big\langle  \left( \begin{matrix} \frac{2\lambda_0^{2\alpha}}{1+\lambda_0} & 0\\
0 &  \frac{2\lambda_0^{1-2\alpha}}{1+\lambda_0} \end{matrix}\right) W^* \xi_0,  W^* \xi_0 \Big\rangle_{\CH_\C}  \text{ (by Eq. \eqref{diagonalize})} \\
&= \half 	\Big\langle  \left( \begin{matrix} \frac{2\lambda_0^{2\alpha}}{1+\lambda_0} & 0\\
0 & \frac{2\lambda_0^{1-2\alpha}}{1+\lambda_0} \end{matrix}\right) (\xi_0+\xi_{0}^\prime),  (\xi_0+\xi_{0}^\prime)\Big \rangle_{\CH_\C}  \text{ (use Eq. \eqref{linearaddition})}\\
&= \half 	\Big\langle  \left( \begin{matrix} \frac{2\lambda_0^{2\alpha}}{1+\lambda_0} & 0\\
0 &  \frac{2\lambda_0^{1-2\alpha}}{1+\lambda_0} \end{matrix}\right) \left( \begin{matrix} 1 \\ 1\end{matrix}\right),   \left( \begin{matrix} 1 \\ 1\end{matrix}\right) \Big\rangle_{\CH_\C}\\ 
&=  \frac{\lambda_0^{2\alpha}}{1+\lambda_0} +\frac{\lambda_0^{1-2\alpha}}{1+\lambda_0}  \\
& =\frac{\lambda_0^{2\alpha} +\lambda_0^{1-2\alpha}}{1+\lambda_0}.
\end{align*}
Finally, note that $\norm{\Delta^z\xi_0}_q=\norm{\Delta^\alpha\xi_0}_q$. 
\end{proof}

\begin{Corollary}\label{xi-norm}
$\norm{\Delta^{\frac{1}{4}} \xi_0}_q= \sqrt{\frac{2\lambda_0^\half}{1+\lambda_0}} $.
\end{Corollary}

\begin{proof}
The proof is immediate from Lemma \ref{normdelta} by putting $ \alpha = \frac{1}{4}$ and $\beta=0$. 
\end{proof}

\begin{Theorem}\label{compactness}
$ \Delta^{\frac{1}{4}}_{\upharpoonleft L^2(M_{\xi_0} , \varphi)} :  L^2(M_{\xi_0},\varphi) \rightarrow  \mathcal{F}_q$ is a Hilbert-Schmidt operator of norm $1$. In particular, 
${\Phi_2}_{\upharpoonleft M_{\xi_0}}:M_{\xi_0}\rightarrow \mathcal{F}_q$ is compact.
\end{Theorem}

\begin{proof}
Let  $\mu = \frac{2\lambda_0^\half}{1+\lambda_0}$.  Since $ \lambda_0 > 1$, so $ \mu <1$. Recall that $\{ \frac{\xi_0^{ \otimes m }}{\sqrt{[m]_q!}} :~  m\geq 0 \}\subseteq M_q\Omega$ $($see Lemma \ref{VectorinMq} and the discussion following it$)$ is an orthonormal basis of $L^2(M_{\xi_0}, \varphi) $. Let $b\in M_{\xi_0}$. Expand 
\begin{align*}
b\Omega = \sum_{m=0}^\infty b_m \frac{\xi_0^{ \otimes m }}{\sqrt{[m]_q!}},
\end{align*}
where $b_m\in\C$ for all $m$ and $\sum_{m=0}^\infty\abs{b_m}^2=\norm{b\Omega}_q^2$. Observe that 
\begin{align*}
\sum_{m=0}^\infty \abs{b_m}^2 \norm{\Delta^{\frac{1}{4}}\frac{\xi_0^{ \otimes m }}{\sqrt{[m]_q!}}}_q^2&=\sum_{m=0}^\infty \abs{b_m}^2 \frac{1}{{[m]_q!}}\norm{({\Delta^\frac{1}{4}\xi_0})^{\otimes m}}_q^2 \text{ (by Eq. \eqref{modulartheory})}\\
&=\sum_{m=0}^\infty \abs{b_m}^2\mu^m \text{ (by Eq. \eqref{Normelt} and Cor. \ref{xi-norm})}\\
&\leq \norm{b\Omega}_q^2.
\end{align*}
Consequently, the series $\sum_{m=0}^\infty b_m \Delta^{\frac{1}{4}}\frac{\xi_0^{ \otimes m }}{\sqrt{[m]_q!}}$ defines a unique element in $\mathcal{F}_q$. Hence, approximating 
$b\Omega$  with $\{\sum_{m=0}^\ell b_m \frac{\xi_0^{ \otimes m }}{\sqrt{[m]_q!}}\}_\ell$, noting that $\Delta^{\frac{1}{4}}$ is closed and using Eq. \eqref{modulartheory}, it follows that 
\begin{align*}
&\Delta^{\frac{1}{4}}b\Omega = \sum_{m=0}^\infty b_m \Delta^{\frac{1}{4}}\frac{\xi_0^{ \otimes m }}{\sqrt{[m]_q!}}, \text{ and, }\\
&\norm{\Delta^{\frac{1}{4}}b\Omega}_q \leq \norm{b\Omega}_q.
\end{align*}
It follows that $\Delta^{\frac{1}{4}}_{\upharpoonleft M_{\xi_0}\Omega}$ admits a bounded extension to $L^2(M_{\xi_0},\varphi)$ and thus is defined on $L^2(M_{\xi_0},\varphi)$.

Further, 
\begin{align*}
\sum_{m=0}^\infty\norm{\Delta^{\frac{1}{4}}(\frac{\xi_0^{\otimes m}}{\sqrt{[m]_q!}})}_q^2&=\sum_{m=0}^\infty\frac{1}{[m]_q!}\norm{(\Delta^{\frac{1}{4}}\xi_0)^{\otimes m}}_q^2 \text{ (by Eq. \eqref{modulartheory})}\\
&=\sum_{m=0}^\infty\mu^m \text{ (by Eq. \eqref{Normelt} and Cor. \ref{xi-norm})}\\
&=\frac{1}{1-\mu}<\infty.
\end{align*}
Therefore, $\Delta^{\frac{1}{4}}_{\upharpoonleft L^2(M_{\xi_0},\varphi)} :L^2(M_{\xi_0},\varphi)\rightarrow \mathcal{F}_q$ is a Hilbert-Schmidt operator. It is obvious that 
$\norm{\Delta^{\frac{1}{4}}_{\upharpoonleft L^2(M_{\xi_0},\varphi)}}=1$.

Consequently, ${\Phi_2}_{\upharpoonleft M_{\xi_0}}$ is compact. Indeed, if $M_{\xi_0}\ni b_n\rightarrow 0$ in the $w^*$-topology, then $b_n\Omega \rightarrow 0$ weakly in  $L^2(M_{\xi_0},\varphi)$. By compactness, $\Delta^{\frac{1}{4}}b_n\Omega\rightarrow 0$ in $\norm{\cdot}_q$. This completes the proof.
\end{proof}

\begin{Theorem}\label{nuclearmapestablished}
${\Phi_2}_{\upharpoonleft M_{\xi_0}}:M_{\xi_0}\rightarrow \mathcal{F}_q$ is a nuclear map.  
\end{Theorem}
\begin{proof}
Following the proof of Thm. \ref{compactness}, it follows that 
\begin{align*}
{\Phi_2}_{\upharpoonleft M_{\xi_0}}(b) &= \Delta^{\frac{1}{4}} b\Omega \\
&= \sum_{m=0}^\infty\langle \frac{\xi_0^{\otimes m}}{ \sqrt{[m]_q!}}, b\Omega \rangle_q\,  \Delta^{\frac{1}{4}} \frac{\xi_0^{\otimes m}}{ \sqrt{[m]_q!}} \\
&= \sum_{m=0}^\infty \psi_m(b) \xi_m, \text{ }b\in M_{\xi_0};
\end{align*}
where $\psi_m\in{(M_{\xi_0})}_* $ is given by $\psi_m(b) =  \langle \frac{\xi_0^{\otimes m}}{ \sqrt{[m]_q!}},~b\Omega\rangle_q $, 
for all $b\in M_{\xi_0}$,  and $\xi_m = \Delta^{\frac{1}{4}} \frac{\xi_0^{\otimes m}}{ \sqrt{[m]_q!}}$, $m\in \mathbb{N}\cup \{0\}$.

By Cauchy-Schwarz inequality, it follows that  $\norm{ \psi_m} \leq 1 $ for all $m\in \mathbb{N}\cup \{0\}$. Therefore, 
\begin{align*}
\sum_{m=0}^\infty \norm{\psi_m}\norm{\xi_m}_q   &\leq	\sum_{m=0}^\infty \norm{\xi_m}_q \\
&=  \sum_{m=0}^\infty \norm{  \Delta^{\frac{1}{4}} \frac{\xi_0^{\otimes m}}{ \sqrt{[m]_q!}}}_q\\
&= \sum_{m=0}^\infty \frac{1}{\sqrt{[m]_q!}} \norm{ (\Delta^{\frac{1}{4}} \xi_0)^{\otimes m}}_q \text{ (by Eq. \eqref{modulartheory})}\\
&= \sum_{m=0}^\infty  (  \frac{2\lambda_0^\half}{1+\lambda_0})^{m/2 } \text{ (by Eq. \eqref{Normelt} and Cor. \ref{xi-norm})}\\
&< \infty \text{  (as }\lambda_0>1).
\end{align*}
Hence, ${\Phi_2}_{\upharpoonleft M_{\xi_0}}$ is a nuclear map $($see Defn. \ref{nuclearmaps}$)$. 
\end{proof}

Now, we extend the above investigation to vectors of the form $0\neq\xi=c_1\xi_0+c_2\xi_0^\prime\in \CH_\R(0)$ with $c_1,c_2\in\R$ and $\abs{c_1}^2+\abs{c_2}^2\leq1$. Calculating as in the proof of Lemma \ref{xi-norm}, one has
\begin{align*}
\langle \frac{2A(0)^{\frac{1}{2}}}{1+A(0)}\xi_0,~\xi_0^\prime\rangle_{\CH_\C}&=\langle WW^*\frac{2A(0)^{\frac{1}{2}}}{1+A(0)}WW^*\xi_0,~\xi_0^\prime\rangle_{\CH_\C}\\
&=\Big\langle \left( \begin{matrix} \frac{2\lambda_0^{\frac{1}{2}}}{1+\lambda_0} & 0\\
0 &  \frac{2\lambda_0^{\frac{1}{2}}}{1+\lambda_0} \end{matrix}\right) W^*\xi_0,W^*\xi_0^\prime\Big\rangle_{\CH_\C} \text{ (use Eq. \eqref{diagonalize})}\\
&=\frac{1}{2}\Big\langle \left( \begin{matrix} \frac{2\lambda_0^{\frac{1}{2}}}{1+\lambda_0} & 0\\
0 &  \frac{2\lambda_0^{\frac{1}{2}}}{1+\lambda_0} \end{matrix}\right) (\xi_0+\xi_0^\prime), ~i(\xi_0^\prime-\xi_0)\Big\rangle_{\CH_\C} \text{ (use Eq. \eqref{linearaddition})}\\
&=\frac{i}{2}\Big\langle \left( \begin{matrix} \frac{2\lambda_0^{\frac{1}{2}}}{1+\lambda_0} & 0\\
0 &  \frac{2\lambda_0^{\frac{1}{2}}}{1+\lambda_0} \end{matrix}\right) \left( \begin{matrix} 1 \\ 1\end{matrix}\right), ~ \left( \begin{matrix}-1 \\ 1\end{matrix}\right)\Big\rangle_{\CH_\C} \text{ (linearity in 2nd variable)}\\
&=0.
\end{align*}
Therefore, if $\xi$ is as stated, then a calculation similar to that in Lemma \ref{xi-norm} shows that 
\begin{align}\label{allalmostperiodic}
\norm{\Delta^{\frac{1}{4}}\xi}_q^2&=\langle \frac{2A(0)^{\frac{1}{2}}}{1+A(0)}\xi,\xi\rangle_{\CH_\C}\\
\nonumber&= \Big\langle \frac{2A(0)^{\frac{1}{2}}}{1+A(0)}(c_1\xi_0+c_2\xi_0^\prime),~(c_1\xi_0+c_2\xi_0^\prime)\Big\rangle_{\CH_\C}\\
\nonumber&=\abs{c_1}^2 \langle \frac{2A(0)^{\frac{1}{2}}}{1+A(0)}\xi_0,~\xi_0\rangle_{\CH_\C}+\abs{c_2}^2 \langle \frac{2A(0)^{\frac{1}{2}}}{1+A(0)}\xi_0^\prime,~\xi_0^\prime\rangle_{\CH_\C}\\
\nonumber&=\frac{2\lambda_0^\half}{1+\lambda_0}(\abs{c_1}^2+\abs{c_2}^2)\\
\nonumber&=\frac{2\lambda_0^\half}{1+\lambda_0}\norm{\xi}_{\CH_\C}^2 \text{ }(=\frac{2\lambda_0^\half}{1+\lambda_0}\norm{\xi}_{U}^2) \\
\nonumber&\leq \frac{2\lambda_0^\half}{1+\lambda_0} <1.
\end{align}

Now suppose that the weakly mixing component of $(U_t)$ is nontrivial, i.e., $\widetilde{\CH}_\R\neq 0$. In this case, we show that for all $\xi\in \widetilde{\CH}_\R$ with $\norm{\xi}_U=1$, we have $\norm{\Delta^{\frac{1}{4}}\xi}_q<1$. 

First fix $\xi\in \CH_\R\cap \mathfrak{D}(A^{-\frac{1}{2}})$ and note that $\xi\in \mathfrak{D}(\Delta^{\frac{1}{2}})$ and hence  $\xi\in \mathfrak{D}(\Delta^{\frac{1}{4}})$. Therefore,
\begin{align}\label{weakmixingnorm}
\norm{\Delta^{\frac{1}{4}}\xi}_q^2& = \langle \Delta^{\frac{1}{4}}\xi, \Delta^{\frac{1}{4}}\xi\rangle_q=\langle \Delta^{\frac{1}{2}}\xi,\xi\rangle_q\\
\nonumber&=\langle A^{-\frac{1}{2}}\xi,\xi\rangle_q ~~\text{ (by Eq. \eqref{modulartheory}, also see \cite[Lemma 1.4]{Hiai})}\\
\nonumber&=\langle \frac{2AA^{-\frac{1}{2}}}{1+A}\xi,\xi\rangle_{\CH_\C}=\langle \frac{2A^{\frac{1}{2}}}{1+A}\xi,\xi\rangle_{\CH_\C}.
\end{align}
Observe that the first and the last expressions in Eq. \eqref{weakmixingnorm} are defined for all $\xi\in\CH_\R$, as $\frac{2A^{\frac{1}{2}}}{1+A}$ is bounded. 

Now, let $\xi\in\CH_\R$ be such that $\norm{\xi}_U=1$. Fix a sequence $\xi_n\in \CH_\R\cap\mathfrak{D}(A^{-\frac{1}{2}})$ such that $\norm{\xi_n}_U=1$ and $\xi_n\rightarrow\xi$
in $\norm{\cdot}_U$ as $n\rightarrow\infty$. Note that $\norm{s_q(\xi_n)}=\norm{s_q(\xi)}=\frac{2}{\sqrt{1-q}}$ for all $n$. Then, $(s_q(\xi_n)-s_q(\xi))\Omega\rightarrow 0$ in $\norm{\cdot}_q$ and hence $(s_q(\xi_n)-s_q(\xi))Jy^*J\Omega \rightarrow 0$ in $\norm{\cdot}_q$ for all $y\in M_q$. A standard density argument forces that $s_q(\xi_n)\rightarrow s_q(\xi)$ in the $\sigma$-strong$^*$ topology $($as $s_q(\xi_n)$, $s_q(\xi)$ are self-adjoint$)$. Combining this with the fact that 
$M_q\ni x\mapsto \Delta^{\frac{1}{4}}x\Omega\in \mathcal{F}_q$ is $\sigma$-strong$^*$ to $\norm{\cdot}_q$ continuous and $\frac{2A^{\frac{1}{2}}}{1+A}$ is bounded, it follows from Eq. \eqref{weakmixingnorm} that 
\begin{align}\label{normcalculated}
\norm{\Delta^{\frac{1}{4}}\xi}_q^2 = \langle \frac{2A^{\frac{1}{2}}}{1+A}\xi,\xi\rangle_{\CH_\C}, \text{ for all }\xi\in \CH_\R. 
\end{align}

Since $\widetilde{\CH}_\R$ is invariant under $(U_t)$ and since we are interested in the behavior of $\Delta^{\frac{1}{4}}$ on $\widetilde{\CH}_\R$, we assume without any loss of generality that $\CH_\R=\widetilde{\CH}_\R$ for the analysis on the weakly mixing component. 

Put $B=\frac{2A^{\frac{1}{2}}}{1+A}$. Then, $0\leq B\leq 1$. Let $B=\int_0^1 \mu d e^B_\mu$ and $A=\int_0^\infty \lambda de^A_\lambda$ denote the spectral resolution of $B$ and $A$ respectively. With respect to the spectral resolution of $A$, the operator $B$ is a direct sum of multiplication by the function $\lambda \rightarrow \frac{2\lambda^{\frac{1}{2}}}{1+\lambda}$ amplified with appropriate multiplicities on classical function spaces. As the spectral measure of $A$ is non-atomic, $B$ has trivial kernel, i.e., $e_0^B=0$. 

Let $f_{\pm}:[0,1]\rightarrow [0,\infty)$ by 
\begin{align*}
f_{\pm}(\mu)=\begin{cases} \frac{(2-\mu^2)\pm 2\sqrt{1-\mu^2}}{\mu^2}, &\mu\in (0,1]\\ 0, &\mu=0.\end{cases}
\end{align*}
By functional calculus, it follows that 
\begin{align*}
&A_1=\int_0^1\lambda de^A_\lambda =\int_0^1 f_{-}(\mu)de^B_\mu, \text{ and, }\\
&A_2=\int_1^\infty \lambda de_\lambda^A = \int_0^1 f_{+}(\mu) de^B_\mu. 
\end{align*}
Note that the ranges of $\chi_{[0,1]}(A)$ and $\chi_{[1,\infty)}(A)$ are orthogonal as the spectral measure of $A$ is non-atomic $($overlap at $\{1\}$ is no issue$)$, and $\chi_{[0,1]}(A)+\chi_{[1,\infty)}(A)=1$, as $A$ is non-singular. Note that $B\in vN(e_\lambda^A: \lambda\in [0,\infty))$, thus $e_\mu^B$ commutes with $e_\lambda^A$ for all $\mu\in [0,1]$ and $\lambda\in [0,\infty)$. 

We claim that the spectral measure of $B$ is non-atomic too. To see this, first let $p_n=\chi_{[1/n, 1]}(B)$, $n\in \N$. If $0<\tau$ is an eigenvalue of $B$, then there exists a unit $($with respect to $\norm{\cdot}_{\CH_\C})$ vector $\delta\in p_n(\CH_\C)$ for some $n\in \N$  such that $B\delta=\tau\delta$. Let $\delta_1= \chi_{[0,1]}(A)\delta$ and $\delta_2=\chi_{[1,\infty)}(A)\delta$. Then, either $\delta_1\neq 0 $ or $\delta_2\neq 0$ or both are nonzero and $\delta=\delta_1+\delta_2$. 

Suppose that $\delta_1\neq 0$. Then, $B\delta_1 =B\chi_{[0,1]}(A)\delta=\chi_{[0,1]}(A)B\delta = \tau\chi_{[0,1]}(A)\delta=\tau\delta_1$.  Therefore, $(p_mf_{-})(B)\delta_1 = (p_mf_{-})(\tau) \delta_1=(p_nf_{-})(\tau) \delta_1=f_{-}(\tau)\delta_1$, for all $m\geq n$. As $\delta_1\in p_n(\CH_\C)$, so $\delta_1\in \mathfrak{D}(A_1)$. Therefore, 
\begin{align*}
A_1\delta_1 = f_{-}(B)\delta_1 =\lim_m (p_mf_{-})(B)\delta_1 = f_{-}(\tau)\delta_1. 
\end{align*}
Consequently, the spectral measure of $A_1$ and hence of $A$ has an atomic component in $(0,1]$. This is a contradiction, and thus $\delta_1=0$. Similarly, by working with $f_+$, it follows that $\delta_2=0$. Therefore, $B$ has no eigenvalues in $(0,1]$. It follows that the spectral measure of $B$ is non-atomic. 

From Eq. \eqref{normcalculated}, Hahn-Hellinger theorem and the fact that the spectral measure of $B^{\frac{1}{2}}$ is non-atomic, it follows that if $\xi\in\widetilde{\CH}_\R$ and $\norm{\xi}_U=1$, then 
\begin{align*}
\norm{\Delta^{\frac{1}{4}}\xi}_q^2 =\norm{B^{\frac{1}{2}}\xi}^2_{\CH_\C} < \norm{\xi}^2_{\CH_\C}= \norm{\xi}^2_{U}=1.
\end{align*}

Now, let $\xi\in\CH_\R$ be such that $\norm{\xi}_{\CH_\C}=\norm{\xi}_U =1$. Following Eq. \eqref{Representation}, let $P_j^1$, $1\leq j\leq N_1$, $P_k^2$, $1\leq k\leq N_2$, and $P_{wm}$ be the orthogonal projections from $\CH_\R$ onto $\R\varsigma_j$, $1\leq j\leq N_1$, $\CH_\R(k)$, $1\leq k\leq N_2$, and $\widetilde{\CH}_\R$ respectively. Then,
\begin{align*}
\xi = \oplus_{j=1}^{N_1} P_j^1 \xi\oplus \oplus_{k=1}^{N_2} P_k^2\xi \oplus P_{wm}\xi.
\end{align*}
Therefore, if $P_k^2\xi\neq 0$ for some $k$ or $P_{wm}\xi\neq 0$, then by Eq. \eqref{allalmostperiodic}
and the preceding discussion, it follows that 
\begin{align*}
\norm{\Delta^{\frac{1}{4}}\xi}^2_q &= \sum_{j=1}^{N_1} \norm{\Delta^{\frac{1}{4}}P_j^1\xi}^2_q + \sum_{k=1}^{N_2} \norm{\Delta^{\frac{1}{4}}P_k^2\xi}^2_q +\norm{\Delta^{\frac{1}{4}}P_{wm}\xi}^2_q \\
&=\sum_{j=1}^{N_1} \norm{P_j^1\xi}^2_q + \sum_{k=1}^{N_2} \norm{\Delta^{\frac{1}{4}}P_k^2\xi}^2_q +\norm{\Delta^{\frac{1}{4}}P_{wm}\xi}^2_q\\
&=\sum_{j=1}^{N_1} \norm{P_j^1\xi}^2_q + \sum_{k=1}^{N_2} \frac{2\lambda_k^{\frac{1}{2}}}{1+\lambda_k}\norm{P_k^2\xi}^2_{U} +\norm{\Delta^{\frac{1}{4}}P_{wm}\xi}^2_q\\
&<\sum_{j=1}^{N_1} \norm{P_j^1\xi}^2_q + \sum_{k=1}^{N_2} \norm{P_k^2\xi}^2_{U} +\norm{P_{wm}\xi}^2_q\\
&=\norm{\xi}_U^2 \\
&=1.
\end{align*}

\begin{Remark}\label{conclusionofmajorthms}
The above analysis entails the following. If $\xi\in\CH_\R$, $\norm{\xi}_U=1$ and $\norm{\Delta^{\frac{1}{4}}\xi}_q<1$, then $M_\xi$ satisfies the conclusions of 
Thm. \ref{compactness} and Thm. \ref{nuclearmapestablished}.
\end{Remark}

Note that, for $p=1,2$, nuclearity of ${\Phi_p}_{\upharpoonleft B}$ and extendability of ${\Phi_p}_{\upharpoonleft B}$ are not usually equivalent $($see Prop. \ref{nuclearimplysplit}$)$. However, in cases that we are interested, the two notions coincide. 

We are now ready to state the main result of this paper. Assuming $dim(\CH_\R)\geq 2$, $N_2\neq 0$ or $\widetilde{\CH}_\R\neq 0$ and following the proof of Thm. \ref{compactness} and Thm. \ref{nuclearmapestablished}, we have:

\begin{Theorem}\label{anyvectoristwocrosstwocomponent}
Let $\xi\in\CH_\R$ be such that $\norm{\xi}_{U}= 1$. Then, the following are equivalent. 
\begin{enumerate}
\item $\xi$ is not fixed by $(U_t)$.
\item $ \Delta^{\frac{1}{4}}_{\upharpoonleft L^2(M_{\xi} , \varphi)} :  L^2(M_{\xi},\varphi) \rightarrow  \mathcal{F}_q$ is a Hilbert-Schmidt operator of norm $1$. In particular, 
${\Phi_2}_{\upharpoonleft M_{\xi}}:M_{\xi}\rightarrow \mathcal{F}_q$ is compact.
\item ${\Phi_2}_{\upharpoonleft M_{\xi}}:M_{\xi}\rightarrow \mathcal{F}_q$ is a nuclear map.
\item $M_\xi\subseteq M_q$ is a quasi-split inclusion.
\end{enumerate}
Suppose $M_q$ is of type $\rm{III}$ and $\xi\in\CH_\R$ satisfies any one of the four equivalent conditions as above. Then, $M_\xi^\prime\cap M_q$ is of type $\rm{III}$. If in addition, $M_q$ is a type $\rm{III}$ factor, then $M_\xi\subseteq M_q$ is a split inclusion and $M_\xi^\prime\cap M_q$ is of type $\rm{III}$.
\end{Theorem}

\begin{proof}
The proof of $(1)\Rightarrow (2)\Rightarrow (3)$ follows from the discussion above. 

\noindent $(3)\Rightarrow (1)$. Suppose to the contrary $U_t\xi=\xi$ for all $t\in\R$. Then, $M_\xi\subseteq M_q^\varphi$ is a MASA in $M_q$ possessing a faithful normal $\varphi$-preserving conditional expectation $\mathbb{E}_\xi$ \cite[Thm. 5.4]{BM17}. Further, since $dim(\CH_\R)\geq 2$, we also have that $M_q$ is a  type $\rm{III}$ factor \cite[Thm. 6.3, Thm. 8.2]{BM17}. From Prop. \ref{nuclearimplysplit} and Prop. \ref{Factorimplysplit}, it follows that $M_\xi\subseteq M_q$ is a split inclusion. Let $F$ be an intermediate type $\rm{I}$ factor between $M_\xi$ and $M_q$. Then, ${\mathbb{E}_\xi}_{\upharpoonleft F}:F\rightarrow M_\xi$ is a faithful normal conditional expectation. This forces that $M_\xi$ is completely atomic, which is a contradiction. Thus, $(1)\Leftrightarrow(2)\Leftrightarrow(3)$.

Note that $(3)\Leftrightarrow(4)$ follows from Prop. \ref{nuclearimplysplit}.

If $M_q$ is of type $\rm{III}$, then the conclusion follows directly from Cor. \ref{largerelativecommutant}. 

If $M_q$ is a type $\rm{III}$ factor, then the conclusion follows directly from Prop. \ref{Factorimplysplit} and Prop. \ref{splitrepresentation}.
\end{proof}

\begin{Remark}\label{rem}
\noindent 
\begin{enumerate}
\item Suppose $\xi\in\CH_\R$ satisfies any one of the four equivalent conditions of Thm. \ref{anyvectoristwocrosstwocomponent}. If $M_q$ has a non-trivial type $\rm{III}$ central summand $M_qp$, then by considering the corner $M_\xi p\subseteq M_q p$, it still follows that $M_\xi^\prime\cap M_q$ is large.
\item It is worth noting that if $\xi$ is as in $(1)$ above, and $M_\xi$ is a MASA in $M_q$, then $M_\xi\subseteq M_q$ is a direct sum of MASAs in type $\rm{I}$ factors $($see Rem. \ref{Quasisplitingeneral}$)$. This particularly applies to the most important case when $dim(\CH_\R)=2$. However, this is not enough to conclude that $M_\xi\subseteq M_q$ is not a MASA since we do not know the type of $M_q$. 
\item The change from presence to absence of $\varphi$-preserving conditional expectation onto $M_\xi$, for $\xi\in \CH_\R$ with $\norm{\xi}_U=1$ $($see \cite[Thm. 4.2]{BM17}$)$, drastically changes the situation. The results also explain that the requirement of presence of conditional expectations onto subalgebras to define solidity and strong solidity is necessary, as without appropriate conditional expectations the relative commutants and hence the normalizing algebras could be beyond control.
\end{enumerate}
\end{Remark}

\noindent\textbf{Acknowledgements}: This work was initiated when the authors were visiting Indian Statistical Institute, Bangalore on Dec 2018. The authors acknowledge the warm hospitality and support of ISI, Bangalore. The second named author thanks Francesco Fidaleo for helpful discussions. Both the authors thank the anonymous referee for pointing to an error and for all the suggestions $($especially with Rem. \ref{Quasisplitingeneral}$)$  leading to better presentation.

\end{document}